\documentclass[11pt, reqno]{amsart}
\usepackage{cite}
\usepackage{pgfplots}
\usepgfplotslibrary{polar}
\pgfplotsset{compat=newest}
\usepackage[margin=1 in]{geometry}
\usepackage{tikz}
\usepackage{tikz-cd}
\usetikzlibrary{matrix,decorations.pathreplacing}
\usepackage{dsfont}
\usepackage{amssymb, amscd,amsmath}
\usepackage[all]{xy}
\usepackage{mathrsfs}
\usepackage{calc}
\usepackage{amsthm}
\theoremstyle{definition}

\newtheorem{theorem}{Theorem}[section]

\newtheorem{lemma}[theorem]{Lemma}

\newtheorem{prop}{Proposition}[section]
\newcommand*\Z{\mathbb{Z}}
\newcommand*\R{\mathbb{R}}

\newcommand{\A}{\mathbb{A}}
\newcommand{\UH}{\mathbb{H}}
\newcommand{\Q}{\mathbb{Q}}

\newcommand{\M}{\mathbb{M}}

\newcommand{\SL}{\mathrm{SL}}
\newcommand{\PSL}{\mathrm{PSL}}

\newcommand{\Pro}{\mathbb{P}}
\newcommand*\CC{\mathbb{C}}

\DeclareMathOperator{\Hom}{Hom}

\begin{document}

\title{Modular Symbols for Fermat curves}
\author{Ozlem Ejder}
\address{Department of Mathematics\\
   Colorado State University\\
  Fort Collins, Colorado 80523}
\email{ejder@math.colostate.edu}

\begin{abstract}
Let $F_n$ denote the Fermat curve given by $x^n+y^n=z^n$ and let $\mu_n$ denote the Galois module of $n$th roots of unity. It is known that the integral homology group $H_1(F_n,\Z)$ is a cyclic $\Z[\mu_n\times \mu_n]$ module. In this paper, we prove this result using modular symbols and the modular description of Fermat curves; moreover we find a basis for the integral homology group $H_1(F_n,\Z)$. We also construct a family of Fermat curves using the Fermat surface and compute its monodromy. %We show that the monodromy action coincides with the action of $\epsilon_0\epsilon_1$ where $\epsilon_0$ and $\epsilon_1$ generate $\mu_n \times \mu_n$.

\end{abstract}
\maketitle

\vspace{1 mm}
\section{Introduction}
  
  Belyi \cite{Belyi} shows that for any smooth projective curve $C/\bar{\Q}$, there exists a finite index subgroup $G$ of 
  $\SL_2(\Z)$ such that the quotient of the upper half plane by the action of $G$ is isomorphic to $C$. As an example, Fermat's cubic curve given by the equation $x^3+y^3=z^3$ is isomorphic to $X_0(27)$; hence in this case $G$ given in Belyi's theorem can be taken as $\Gamma_0(27)$. Similarly Fermat's quartic $x^4+y^4=z^4$ is isomorphic to the modular curve $X_0(64)$ \cite{Kenku64}. A natural question arises here: \textit{Given $n$, can we describe $G_n$; a subgroup of $\SL_2(\Z)$  such that that $X_{G_n}$ is isomorphic to the Fermat curve $F_n: x^n+y^n=z^n$?} This is done in \cite{Rohrlich}; Rohrlich gives a modular description of Fermat curves in terms of certain subgroups so called Fermat groups $\Phi(n)$ of $\Gamma(2)$. Using this description he computes the rank of the subgroup of the Mordell-Weil group $J(F_n)(\bar{\Q})$ generated by divisors supported on cusps (i.e., points $(x:y:z)$ with  $xyz=0$). 
      
The homology of modular curves can be described using modular symbols \cite{Manin}. A modular symbol is the image of a geodesic from $\alpha$ to $\beta$ in $X_G(\CC)$ for some $\alpha,\beta$ in $\Pro^1(\Q)$ and it is denoted by $\{\alpha,\beta\}$. Modular symbols are used as a tool to compute modular forms \cite{Stein} where the action of the Hecke algebra is more explicit. We will use modular symbols to describe the homology of Fermat curves.

The homology of Fermat curves has been computed by various people. Let $\epsilon_0,\epsilon_1$ be the automorphisms of $F_n$ given by 
\[ \epsilon_0:  (x:y:z)\mapsto (\zeta_n x: y:z), \hspace{2mm} \epsilon_1:  (x:y:z)\mapsto ( x: \zeta_ny:z),
\]
where $\zeta_n$ denotes a primitive $n$th root of unity. Let $\mu_n\times \mu_n$ be the subgroup of the automorphism group of $F_n$ generated by $\epsilon_0$ and $\epsilon_1$. In the appendix to the paper by Gross \cite{Gross}, Rohrlich showed that $H_1(F_n,\Z)$ is a cyclic $\Z[\mu_n\times \mu_n]$ module generated by the Pochhammer contour $\kappa=-(1-\epsilon_0)(1-\epsilon_1)\gamma$ where $\gamma :[0,1] \to F_n(\CC)$ is given by $t \mapsto (t^{1/n}:(1-t^n)^{1/n}:1)$. 
  In \cite[Proposition~B]{Lim}, Lim carried out this calculation and computed a basis $\{ \epsilon_0^i\epsilon_1^jg \mid 0\leq i \leq n-2, \hspace{2mm}   0\leq j \leq n-3\}$ for $H_1(F_n,\Z)$ using the generator $g$ described as: 
\begin{equation}\label{eq:g}
    g= \left\{\begin{array}{lr}
       \epsilon_0^{(n-1)/2} \epsilon_1^{(n-1)/2}(1-\epsilon_0)(1-\epsilon_1)\gamma, & \text{for }  \text{n odd} \\\\
       (1-\epsilon_0^{n-1})(1-\epsilon_1^{n-1})\gamma, & \text{for } \text{n even}\\
            \end{array}\right\}. 
              \end{equation}
 Notice here that the image of $\gamma$ remains in the affine curve $U(\CC)$ given by $x^n+y^n=1$.

In \cite{Anderson}, Anderson (using motivic homology) studied the relative homology (with $\Z/n$ coefficients) of the affine Fermat curve $U: x^n+y^n=1$ with respect to cusps $\{(x,y) \in U(\CC) \mid xy=0\}$. He showed that $H_1(U,\{xy=0\};\Z/n)$ is a free $\Z/n[\mu_n\times \mu_n]$ module generated by $\gamma$ which we described above. See also \cite{DPSWwin} for another description of the homology of $F_n$.   
  In this paper we take a new approach. Using Rohrlich's description of Fermat curves, we compute the modular symbols for the Fermat group $\Phi(n)$ which equivalently gives us the basis for the homology group $H_1(F_n,\{xyz=0\};\Z)$ of the Fermat curve $x^n+y^n=z^n$ relative to the set of cusps and also allows us to compute a basis for the integral homology of $F_n$. Our main results are the following:
     
  Let $\phi$ denote the automorphism of $F_n$ given by $(x:y:z) \mapsto (z:\epsilon x:y)$ where $\epsilon$ is a primitive $2n$th root of unity (namely $e^{\pi/n}$ ) and let $\bar{\gamma}$ denote the path we obtain by the action of $\phi$ on $\gamma$, i.e., $\bar{\gamma} : [0,1] \to F_n(\CC)$ is given by $t \mapsto (1:\epsilon(1-t^n)^{1/n}:t^{1/n} )$.
  \begin{theorem} \label{thm:modular symbols}\hfill
  \begin{enumerate}
  \item The relative homology group $H_1(F_n,\{xyz=0\};\Z)$ is generated by $\gamma$ and $\bar{\gamma}$ as a $\Z[\mu_n\times \mu_n]$ module with the relation 
  \begin{align} \label{relation e}
   \epsilon_0\gamma+  \epsilon_0\epsilon_1\bar{\gamma}= \gamma+ \bar{\gamma}.
  \end{align}
  \item A basis of  $H_1(F_n,\{xyz=0\};\Z)$ is given by the elements
       \[  \{  \epsilon_0^i\epsilon_1^j\bar{\gamma} \mid 1\leq i \leq n-1, 0\leq j \leq n-1\},  \hspace{2mm}
              \{  \epsilon_0^{n-1}\epsilon_1^j\gamma \mid 0\leq j \leq n-1 \},
                 \hspace{1mm}  \text{and}  \hspace{2mm} \{ \epsilon_1\bar{\gamma}\}. 
                  \]
\end{enumerate}
  \end{theorem}
  
  Here the image of $\gamma$ remains in the affine curve $F_n \cap  (z\neq 0)$ whereas the image of $\bar{\gamma}$ lives in the affine curve $F_n \cap (x\neq 0)$. The relation we determine between $\gamma$ and $\bar{\gamma}$ describes how the relative homologies of these affine curves patch in $H_1(F_n, \{xyz=0\},\Z)$.
  
  \begin{theorem} \label{thm:cuspidal} 
  A basis for the integral homology $H_1(F_n,\Z)$ of $F_n$ is given by the following elements:
   \[
   \{ \epsilon_0^i\epsilon_1^j (1-\epsilon_0)(1-\epsilon_1)\gamma \hspace{2mm} \mid 1 \leq i \leq n-2, \hspace{2mm} 0\leq j\leq n-2 \}
  \]
  \end{theorem} 
 Moreover, the relation in \eqref{relation e} gives us that $(1+\epsilon_j +\epsilon_j^2+\ldots+\epsilon_j^{n-1} )$ for $j=1,2$ and $(1+\epsilon_0\epsilon_1 +\epsilon_0^2\epsilon_1^2+\ldots+\epsilon_0^{n-1}\epsilon_1^{n-1} )$ in $\Z[\mu_n\times \mu_n]$ operate as zero on $H_1(F_n, \Z)$.

  Understanding the $\mu_n\times \mu_n$ action on the homology groups is important for many reasons.  In section~\ref{Fermat surfaces}, we give an example of such a reason; considering the rational map from the Fermat surface $x^n+y^n+z^n=w^n$ to $\Pro^1$ given by $(x:y:w:z) \mapsto (w:z)$, we construct a fibration where the fibers are isomorphic to Fermat curves. We show that the monodromy action on the homology of the fibers of this family is induced by the automorphism $\epsilon_0\epsilon_1:(x:y:z) \mapsto (\zeta_nx:\zeta_ny:z)$. In section~\ref{monodromy}, we determine the action of $\epsilon_0$ and $\epsilon_1$ on the generators of the homology given in Theorem~\ref{thm:cuspidal} which describes the $\mu_n\times \mu_n$ action on $H_1(F_n,\Z)$ (Theorem~\ref{thm:action}).  

%%%%%%%%%%%%%%%%%%%%%%%%%%%%%%%%%%%%%%%%%%%
\section{Acknowledgements}
We thank Rachel Pries and the anonymous referee for the valuable comments on an earlier version of this article.
%------------------------------------------------------------------------------------
\section{Background}\label{background}
%%%%%%%%%%%%%%%%%%%%%%%%%%%%%%%%%%  \gamma 
  \subsection{Modular Interpretation of Fermat curves}\label{Fermat}
%%%%%%%%%%%%%%%%%%%%%%%%%%%%%%%%%%

      Let $\bar{\mathfrak{H}}$ denote the extended upper half plane  $\mathfrak{H}\hspace{.5mm} \cup \hspace{.5mm}\Pro^1(\Q)$.  An element $\begin{bmatrix} a&b\\ c&d \end{bmatrix}$ of $\SL_2(\Z)$ acts on $\bar{\mathfrak{H}}$ via fractional linear transformations:
       \[ 
     \begin{bmatrix} a&b\\ c&d \end{bmatrix}   .z= \frac{az+b}{cz+d}
        \]
 Since the action of $SL_2(\Z)$ on  $\bar{\mathfrak{H}}$ factors through $\PSL_2(\Z)$, we abuse the notation and denote a subgroup $G$ of $\SL_2(\Z)$ and its image in $\PSL_2(\Z)$ both by $G$.
  
  Let $G$ be a finite index subgroup of $\PSL_2(\Z)$. The orbit space $X_{G}(\CC)=\bar{\mathfrak{H}}/G$ has a natural structure of a smooth compact complex space of topological dimension $2$ and hence also a structure of an irreducible projective algebraic curve. 
  
  We call a point on $X_{G}$ cuspidal (or a cusp) if it is in the image of the subset $\Pro^1(\Q)$ of $\bar{\mathfrak{H}}$. We also denote the point at infinity on $\bar{\mathfrak{H}}$ by $i\infty$.
   
 In the following, we describe $\Phi(n)$, subgroups of $\PSL_2(\Z)$ such that  for every integer $n \geq1$, the curve $X_{\Phi(n)}$ is isomorphic to the Fermat curve                   
                   \[ 
                    x^n+y^n=z^n.
                     \]
 We first describe the principle congruence subgroup $\Gamma(2)$:
                      \[
                   \Gamma(2)= \left\{\begin{bmatrix} a&b\\ c&d \end{bmatrix} \in \SL_2(\Z): a,d \equiv 1 
                   \hspace{1mm} \text{mod} \hspace{1mm} 2 \hspace{1mm} \text{and} \hspace{1mm} b,c 
                   \equiv 
                0 \hspace{1mm} \text{mod} \hspace{1mm} 2 \right\}. 
                   \]
   It is well known that $\Gamma(2)$ is generated by $A$ and $ B$ (inside $\PSL_2(\Z)$) where
                  \[
                   A:=\begin{bmatrix} 1&2\\
                  0&1  \end{bmatrix}  \hspace{2mm} \text{and} \hspace{2mm} B:=\begin{bmatrix} 1&0\\
                  2&1  \end{bmatrix}.
                  \]          
      Let $\Phi(n)=\langle A^n, B^n, \Gamma(2)' \rangle$ be the subgroup of $\Gamma(2)$ generated by 
   $A^n, B^n$, and the commutator subgroup $\Gamma(2)'$ of $\Gamma(2)$. Then the inclusion
   $\Phi(n)   \hookrightarrow \Gamma(2) $ of groups induces a morphism of curves:
                   \[ 
                      X_{\Phi(n)} \longrightarrow X_{\Gamma(2)}. 
                   \]
     Rohrlich proves in \cite{Rohrlich} that the modular curve $X_{\Phi(n)}$ is isomorphic to the Fermat curve 
              \[   
              F_n : x^n+y^n=z^n.
              \]
    It is also known that the modular curve $X_{\Gamma(2)}=X(2)$ is a genus $0$ curve. The isomorphism $X(2) \to \Pro^1$ is given by lambda function $\lambda$ and it has the following values at the cusps $0,1$ and $i\infty$ of $X(2)$ (\cite[pg,116]{Chandrasekharan}):
      \[ 
      \lambda(0)=1, \hspace{2mm}  \lambda(1)=\infty , \hspace{2mm}  \lambda(i\infty)=0 
      \]
    Here the $\lambda$ values $1,\infty$ and $0$ represent the points $[1:1],  [1:0], [0:1]$. Moreover, the map $ X_{\Phi(n)} \longrightarrow  X(2)$ corresponds  to the projection map (see \cite[Section 5.5, pg~16]{Long}):
                      \[
                       g: F_n \longrightarrow \Pro^1 
                      \]
                       \[ 
                       (x:y:z) \mapsto (x^n:z^n) 
                       \]                   
   The morphism $g$ has degree $n^2$ with ramification above the points $[0:1],[1:1]$ and $[1:0]$ of $\Pro^1$. Above each of the points $[0:1],[1:1]$ and $[1:0]$, the morphism $g$ has ramification degree $n$, hence we can conclude that the modular curve $X_{\Phi(n)}$ has $3n$ cusps and there are exactly $n$ of them lying above each of the points $[0:1],[1:1]$ and $[1:0]$ in $\Pro^1$. 
   
To summarize, the cuspidal points on $F_n$; the points in the image of the cusps of $X_{\Phi(n)}$, are the points $[x:y:z]$ with $xyz=0$ which are known as the points at infinity \cite{Rohrlich}.

%---------------------------------------------------------------------------------- 
     
  \subsection{Cusps of {$X_{\Phi(n)}$}}\label{cusps}
 %----------------------------------------------------------------------------------
   
     Define $a_k,b_k$, and $c_k$ for $1\leq k\leq n$ as:
      \[
       a_k=(0:\zeta^k:1), \hspace{2mm} b_k=(\zeta^k:0:1), \hspace{2mm} c_k=(\epsilon\zeta^k:1:0),
       \]
      where $\epsilon=e^{{\pi i}/n}$ and $\zeta$ is a primitive $n$'th root of unity.  We will show that under the isomorphism between $X_{\Phi(n)}$ and $F_n$, the subset $\{ A^j.1 \mid j=1,\ldots,n\}$ of $X_{\Phi(n)}(\CC)$ maps to the subset $\{c_k \mid k=1,\ldots,n\}$ of $F_n(\CC)$ using the following lemma.
      
       \begin{lemma}\label{lem:actionon1}
         Let $\pi$ denote the quotient map $\bar{\mathfrak{H}} \rightarrow  \bar{\mathfrak{H}}/ \Phi(n)$. 
         Then 
         \begin{enumerate}
         \item  The action of $A$ and $B$ on $X_{\Phi(n)}$ commute.     \label{commute}          
          \item  We have \label{A=B}
        $ \pi(A^k.1)=\pi(B^k.1)$ for each $k \in \Z$.

                   \end{enumerate}
    \end{lemma} 
    \begin{proof} 
    
       The part~(\ref{commute}) follows from the fact that $\Phi(n)$ contains the commutator subgroup of $\Gamma(2)$ and that the action of $\Gamma(2)$ on $X_{\Phi(n)}$ factors through $\Gamma(2)/\Phi(n)$. We first note that if $\gamma_1, \gamma_2$ are in $\Gamma(2)$, then we have               
             \[
              \pi(\gamma_1\gamma_2.z)=\pi(\gamma_2\gamma_1.z)
              \]
       by part (\ref{commute}).
       
  We will prove the second part by induction. Assume $k=1$. We easily compute that $BA^{-1}.1=1$ and $(ABA^{-1}B^{-1})B.1=A(BA^{-1}.1)=A.1$. This shows that
        \[  
                \pi(A.1)=\pi(B.1)
         \]  
 since $ABA^{-1}B^{-1}$ is in $\Phi(n)$.  Assume 
            \[
            \pi(A^{k-1}.1)=\pi(B^{k-1}.1)
            \] 
     for some $k$. Remember that $\pi(z)=\pi(z')$ if and only if there exist a $\gamma \in \Phi(n)$ such that 
     $z'=\gamma z$. Then 
            \begin{align*} 
                      \pi(A^k.1)=\pi(A.(A^{k-1}.1))&=\pi(A(\gamma B^{k-1}).1)= \pi((\gamma .B^{k-1})A.1)=
                                 \pi( \gamma' B^k.1),
             \end{align*}   for some $\gamma, \gamma'$ in $\Phi(n)$. 
     \end{proof}

Since the automorphism group of the morphism $X_{\Phi(n)} \to X(2)$ is $\Gamma(2)/\Phi(n)$, the orbits of the points $A^iB^j.1$ for $i,j=1,\ldots,n$ will give us all the points on $X_{\Phi(n)}$ that are lying above the orbit of $\tau=1$ 
on $X(2)$. By Lemma \ref{lem:actionon1}, the set of points in $X_{\Phi(n)}$ lying above $\tau=1$ is given by the images of $A^i.1$.  Hence the set $\{c_k\}$ corresponds to the set $\{A^j.1\}$ in $X_{\Phi(n)}$ since $A^iB^j.1=A^i.gA^j.1=g'A^{i+j}$ for some $g,g' \in \Phi(n)$.
     
   %-----------------------------------------------------------------------------------    
  
   \subsection{Automorphisms of {$X_{\Phi(n)}$}}\label{Automorphisms}
   
   It is proven in \cite{Tzermias} that for $n\geq 4$, the automorphism group of the Fermat curve $x^n+y^n=z^n$ is 
   generated by two automorphisms of order $n$ which are given by 
               \[ 
             \epsilon_0:  (x:y:z)\mapsto (\zeta_n x: y:z),
               \] 
               \[ 
             \epsilon_1:  (x:y:z) \mapsto (x: \zeta_n y:z)
               \]
    and the permutation group $S_3$ on three letters. On the other hand, the normalizer of $\Phi(n)$ in 
    $\SL_2(\Z)$ acts on $X_{\Phi(n)}$. 
    
    We first note that $B^{-1}A$ is in the normalizer of $\Phi(n)$ and the sets $\{A^i.0\}$ and $\{B^j.i\infty\}$ are stabilized under the action of $B^{-1}A$. Moreover, the orbit of 
    $A^j.1$ in $X_{\Phi(n)}$ is fixed for each $j$ since $\pi((B^{-1}A).(A^j.1))=\pi(A^{j+1}.(B^{-1}.1))=\pi(A^j.1)$.
    
    The automorphisms of order 
    $n$ of $F_n$ fixing $c_k$ for each $k$ and stabilizing the sets $\{a_i\}, \{b_j\}$ is generated by
         \[ 
         \epsilon_0\epsilon_1 :(x:y:z) \mapsto (\zeta_n x: \zeta_n y: z).
         \] 
    Hence the automorphism $ \epsilon_0\epsilon_1$ corresponds to the action of some power of $B^{-1}A$ on the modular curve $X_{\Phi(n)}$. We will fix a primitive $n$th root of unity $\zeta_n$ such that the action of $A$ is given by $\epsilon_0$ and the action of $B^{-1}$ is given by $\epsilon_1$.

 %---------------------------------------------------------------------------------------------- 
%--------------------------------------------------------------------------------------  
   \subsection{Modular Symbols} \label{sec:Manin}
   
  This section is a summary of Manin's results from \cite{Manin} and \cite{Stein}.
   
  Let $G$ be a finite index subgroup of $\SL_2(\Z)$ and let $\alpha,\beta$ be in $\bar{\mathfrak{H}}$. If $\pi(\alpha)=\pi(\beta) \in X_G(\CC)$, then any path from $\alpha$ to $\beta$ on $\bar{\mathfrak{H}}$ becomes a closed path on $X_G(\CC)$ whose homology only depends on the end points $\alpha,\beta$. We denote this homology class by the symbol $\{\alpha,\beta\} \in H_1(X_G(\CC),\Z)$. More generally for any $\alpha,\beta \in \bar{\mathfrak{H}}$, we define the modular symbol $\{\alpha,\beta\}$ as an element of $\Hom_{\CC}(H^0(X_G(\CC),\Omega^1),\CC)$:
  \[
  \{\alpha,\beta\}: ( \omega \mapsto  \int_{\alpha}^{\beta}{\pi^*\omega}). 
  \]
 Using the $\R$-isomorphism 
  \[ 
 H_1(X_G(\CC),\R) \to \Hom_{\CC}(H^0(X_G(\CC),\Omega^1),\CC).
 \]
 given by $\gamma \mapsto ( \omega \mapsto  \int_{\gamma}{\pi^*\omega})$, we view  $\{\alpha,\beta\}$ as a real homology class in $H_1(X_G(\CC),\R)$.
 
  Define a left action of $\SL_2(\Z)$ on the set of modular symbols by letting $g \in SL_2(\Z)$ 
       act by 
               \[ 
                g\{\alpha,\beta\}=\{g\alpha,g\beta\}
                 \]
       where the action of $g=\begin{bmatrix} a&b\\c&d \end{bmatrix}$ on $\alpha$ and $\beta$ is as it is 
       defined in Section~\ref{Fermat}.
Modular symbols have the following properties:
 \begin{enumerate}
 \item $\{\alpha,\beta\} + \{\beta, \gamma\} +\{\gamma,\alpha\}=0$.
 \item $g\{\alpha,\beta\}=\{g\alpha,g\beta\}=\{\alpha,\beta\}$ for all $g\in G$.
 \end{enumerate}
          
     Suppose $a \in \SL_2(\Z)$. Then we call the modular symbol $a \{0,i\infty\}$ the Manin symbol associated to $a$. Notice here 
       that if the right cosets $Ga$ and $Ga'$ are equal, then $a\{0,i\infty\}=a'\{0,i\infty\}$ in $X_G(\CC)$. Hence, there is a well-defined 
       Manin symbol associated to each right coset of $G$ in $\SL_2(\Z)$. We will use the notation $[a]$ 
       for the Manin symbol of the coset $Ga$. (For more details, one can read \cite{Manin} or \cite{Stein}.)
    We can describe the integral homology of $X_G$ using Manin symbols.
    
      \begin{theorem}[\cite{Manin}] \label{Manin2} Any class $h$ in $ H_1(X_{G},\Z)$ can be 
        represented as a sum $\sum_k{n_k[g_k]}$ of Manin symbols $[g_k]=\{\alpha_k,\beta_k\}$ and this representation can be chosen so that 
        \[ 
        \sum_k{n_k(\pi(\beta_k) -\pi(\alpha_k))} = 0
        \] 
                  as a zero dimensional cycle on $X_{G}(\CC)$ where $\pi: \bar{\UH} \to X_G(\CC)$ is the natural projection map. 
         \end{theorem} 
  Manin in fact shows that every modular symbol can be written as a $\Z$-linear combination 
       of Manin symbols and moreover he determines the relations between these generators. Let
                     \[ 
                     \sigma=\begin{bmatrix} 0     &1\\
                                             -1      & 0\end{bmatrix}, \hspace{2mm}
                            \tau=\begin{bmatrix} 0  &-1 \\ 
                                                  1 & -1  \end{bmatrix},  \hspace{2mm} 
                         J=\begin{bmatrix}     -1   &0     \\ 
                                                 0 & -1\end{bmatrix}.
                     \]
          \begin{theorem}[\cite{Manin}] \label{Manin1}
          
           If $[a]$ is a Manin symbol, then
                      \begin{align} 
                      [a]+[a\sigma] &= 0 \label{sigma}\\
                            [a]+[a\tau] +[a\tau^2] &=0 \label{tau}\\
                            [a]-[aJ] &= 0.
                      \end{align}
          \end{theorem}
    (Since the homology group $H_1(X_G,\R)$ is torsion free, it follows that if $[a]=[a\sigma]$ or $[a]=[a\tau]$, then $[a]=0$.)     
 
 Manin states that in some sense, this system of relations is complete and formulates this precisely by the following argument (\cite[1.8-1.9]{Manin}):
  
  Let $C$ be the group generated by the symbols $(a)$ where $a$ is a right coset of $G$ in $\PSL_2(\Z)$ with the relations $(a)+(a\sigma)=0$ for all $a$ and $(a)=0$ if $(a)=(a\sigma)$. Define a boundary map $\partial :C \to \langle \text{cusps}\rangle$ by $(a) \mapsto \pi(a.i\infty)-\pi(a.0)$ where  $\langle\text{cusps}\rangle$ is the free abelian group generated by the cusps of $X_G$. We designate the kernel of $\partial$ by $Z$. Let $B$ be the subgroup of $C$ generated by the elements $(a)$ such that $(a)=(a\tau)$ and the elements $(a)+(a\tau)+(a\tau^2)$ for the remaining $a$. Manin proves that the quotient group $Z/B$ is isomorphic to $H_1(X_G,\Z)$ and $C/B$ is isomorphic to the relative homology group $H_1(X_G,\{\text{cusps}\},\Z)$ using a fairly involved topological construction.

We will use Stein \cite{Stein}'s notation for the group of modular symbols and denote it as $M_2(G)$.
 %-----------------------------------------------------------------------------------------

\section{Modular Symbols for Fermat Curves}\label{modular fermat curves}
%---------------------------------------------------------------------------------------------

  In this section, we compute the group of modular symbols $M_2(\Phi(n))$ for $\Phi(n)$ which allows us to prove Theorem~\ref{thm:modular symbols}.
    \begin{prop} \label{prop:modular symbols}
               The group of modular symbols for the Fermat group $\Phi(n)$ is free of rank $n^2+1$ and 
                it is generated by 
                  \[ 
                  \{ [A^iB^j\tau] : 1\leq i \leq n-1, 0\leq j \leq n-1 \} 
                  \]
                  \[ 
                   \{[A^{n-1}B^j] : 0\leq j \leq n-1    \} 
                   \]
                      and
                  \[ 
                  [B^{n-1}\tau]. 
                  \]
    \end{prop}

 In order to prove Proposition~\ref{prop:modular symbols}, we first find the Manin symbol associated to each right coset of $\Phi(n)$ in $\PSL_2(\Z)$ and compute the relations given in Theorem~\ref{Manin1}.
    
 \subsection{Cosets of $\Phi(n)$} 
  We will begin with finding coset representatives of 
  $\Phi(n)$ in $\PSL_2(\Z)$. It is easy to see that the cosets of $\Phi(n)$ in $\Gamma(2)$ are 
  given by $A^iB^j$ where $0 \leq i,j\leq n-1.$ Notice that the action of $M$ and $-M$ coincides for $M\in \SL_2(\Z)$ and hence we consider the cosets in $\PSL_2(\Z)$ here instead of $\SL_2(\Z)$.
       
        \begin{lemma} \label{coset}
           A complete set of right coset representatives of $\Gamma(2)$ in $\PSL_2(\Z)$ consists of the 
           following matrices:
                \[
                \alpha_0= \begin{bmatrix} 1&0\\0&1 \end{bmatrix},
               \alpha_1= \begin{bmatrix}    1&1\\-1&0 \end{bmatrix},
                \alpha_2=  \begin{bmatrix} 0&-1\\1&-1 \end{bmatrix},
                \alpha_3=  \begin{bmatrix} 0&1\\-1&0 \end{bmatrix},
                \alpha_4=  \begin{bmatrix} 1&1\\0&1 \end{bmatrix},
                 \alpha_5= \begin{bmatrix} 1&0\\1&1 \end{bmatrix}.
                 \]
        \end{lemma}
           \begin{proof} 
           One can show by direct computation that $\alpha_i\alpha_j^{-1} \not\in \Gamma(2)$ for any 
           $i\neq j$. Since the index of $\Gamma(2)$ in $\PSL_2(\Z)$ is $6$, we are done.
           \end{proof}

     By Lemma \ref{coset}, the set of right cosets of $\Phi(n)$ in $\PSL_2(\Z)$ is given by
         \[ 
         \{ A^iB^j\alpha_k: \hspace{2mm} 0\leq i,j \leq (n-1) \hspace{1mm} \text{and} \hspace{1mm} 0\leq 
         k\leq 5\}.
         \]     
 Hence the set $\{[A^iB^j\alpha_k]\}$ generates the group of modular symbols. To be able to find free generators for this group, we will compute the relations given in Theorem \ref{Manin1}. 

We note again that $\alpha_3=\sigma=\begin{bmatrix} 0&1\\-1&0 \end{bmatrix}$ and $\alpha_2=\tau= \begin{bmatrix} 0&-1\\1&-1 \end{bmatrix}$.

      The matrices $A^n$ and $B^n$ act trivially on the modular symbols since they are in $\Phi(n)$. Hence in the following computations and for the rest of the paper, the powers of $A$ and $B$ are considered modulo $n$. Also remember that $[a]$ denotes the Manin symbol of $a \in \PSL_2(\Z)$.
     %*****************************************************************************  
   
   \subsection{$\sigma$-relations}\label{sigma relations}
 %****************************************************************************      
      We first compute that
           \[ 
            \alpha_1 \sigma =A^{-1}J\alpha_4, \quad
                        \alpha_2 \sigma =\alpha_5,\quad
                        \alpha_3 \sigma =\alpha_0 J=J,
           \] 
           \[
                         \alpha_4 \sigma =A\alpha_1,\quad
                         \alpha_5 \sigma=J \alpha_2,\quad
                         \alpha_0 \sigma =\sigma.
            \]
      Thus we have
         
          \[ 
            [A^iB^j \alpha_1 \sigma] =[A^{i-1}B^j\alpha_4],  \quad
                        [A^iB^j\alpha_2 \sigma] =[A^iB^j\alpha_5], \quad
                       [A^iB^j\alpha_3 \sigma] =[A^iB^j],
           \] 
           \[
                      [A^iB^j \alpha_4 \sigma] =[A^{i+1}B^j\alpha_1], \quad
                       [A^iB^j\alpha_5 \sigma]=[A^iB^j \alpha_2], \quad
                       [A^iB^j \sigma] =[A^iB^j\alpha_3].
            \]
     Notice that the right cosets of $\Phi(n)$ in $\Gamma(2)$ commute since $\Phi(n)$ contains the 
     commutator subgroup of $\Gamma(2)$. Now using Equation (\ref{sigma}), we have the following 
     equalities:
            
              \begin{align} 
              \label{4.4} [A^iB^j \alpha_1] + [A^{i-1}B^j \alpha_4] &=0 \\
                     [A^iB^j\alpha_2] + [A^{i}B^j\alpha_5] &=0 \\
                     [A^iB^j\alpha_3] + [A^{i}B^j]         &=0 . 
              \end{align}
      Therefore
      the group of modular symbols for $\Phi(n)$ is generated by 
              \[ 
               \{ [A^iB^j\alpha_k]: \hspace{1mm} 0\leq i,j\leq (n-1) \hspace{1mm} \text{and} \hspace{1mm} 
               k=0,1,2\}.
               \] 

%----------------------------------------------------------------------------------------     
   \subsection{{$\tau$}-relations}\label{tau relations}
 %*************************************************************************
      Similarly we compute $\alpha_k\tau$ for $k=0,\ldots, 5$ as the following:
       
        \[
         \alpha_1\tau=A^{-1},    \hspace{2mm}  \alpha_2 \tau= A\alpha_1,  \hspace{2mm}  
         \alpha_3 \tau= A^{-1} \alpha_4,
        \]
        \[ 
        \alpha_4 \tau=AB^{-1}\alpha_5J,  \hspace{2mm}   \alpha_5 \tau=B\alpha_3J,     \hspace{2mm}
        \alpha_0 \tau=\tau.
        \]
      Then we find that the Manin symbol of $A^iB^j\alpha_k\tau$ as:
        
        \[
         [A^iB^j \alpha_1\tau]=[A^{i-1}B^j],    \hspace{2mm}  
         [A^iB^j\alpha_2 \tau]= [A^{i+1}B^j\alpha_1],  \hspace{2mm}  
         [A^iB^j\alpha_3 \tau]= [A^{i-1}B^j \alpha_4],
        \]
       
        \[ 
        [A^iB^j\alpha_4 \tau]=[A^{i+1}B^{j-1}\alpha_5],  \hspace{2mm}   
        [A^iB^j\alpha_5 \tau]=[A^iB^{j+1}\alpha_3],     \hspace{2mm}
        [A^iB^j\tau]=[A^iB^j\alpha_2].
        \]
      Hence by Equation \ref{tau} we obtain:
      
              \begin{align} \label{tau1}
              [A^iB^j]+[A^iB^j\alpha_2 ]   +[A^{i+1}B^j\alpha_1]& =0 \\ \label{tau2}                                                              
              [A^iB^j\alpha_3]  + [A^{i-1}B^j\alpha_4] + [A^{i}B^{j-1}\alpha_5] & =0.
              \end{align}
      Now using the results of section \ref{sigma relations}, Equation \ref{tau2} becomes     
            \begin{equation}\label{tau-new} 
                        [A^iB^j]+[A^iB^j\alpha_1] +   [A^{i}B^{j-1}\alpha_2] =0   
            \end{equation}
       Hence by Equation \ref{tau1}, we conclude that the set of Manin symbols 
         \[
         \{ [A^iB^j] \hspace{1mm} \text{and} \hspace{1mm} [A^iB^j \alpha_2] : \hspace{1mm} 0 \leq i,j\leq 
         (n-1). \}
         \] 
         generate $M_2(\Phi(n))$. Therefore using Equation \ref{tau-new} and Equation \ref{tau1}, we 
         obtain the following relation:
             \begin{equation} \label{relation}
                         [A^{i+1}B^{j}] + [A^{i+1}B^{j-1}\tau]=[A^iB^j] +                                                                                 
                             [A^iB^j\tau]  
             \end{equation} 
                 for $0 \leq i \leq (n-1)$ and $0 \leq j \leq n-1$. Remember that $\alpha_2=\tau$.

      \subsection{Proof of Proposition~\ref{prop:modular symbols}}
               We will use $x_{i,j}:=[A^iB^j]$ and $y_{i,j}:=[A^iB^j\tau]$ to ease the notation. We put 
               Lexicographical order on the 
               set $\{x_{i,j},y_{i,j}:  i,j=0,\ldots n-1\}$ as follows: \\
                   For any $i,j,k,l$,
                   \[
                    x_{i,j} \leq y_{k,l}
                   \] 
                   and
                  \[  
                  x_{i,j}  \leq   x_{k,l}       \hspace{2mm}     \text{or}     \hspace{2mm} 
                  y_{i,j}  \leq   y_{k,l}       \hspace{2mm} 
                    \text{if and only if}      \hspace{2mm}      
                     i < k    \hspace{2mm} \text{or} \hspace{2mm} i = k 
                   \hspace{2mm} \text{and} \hspace{2mm}  j \leq l. 
                  \]
                Also note that the index set for $i,j$ is $\Z/n\Z$. Hence we see that if $i=n-1$ and $j=0$, then 
                $x_{i+1,j-1}=x_{0,n-1}$. 
                
                We already showed that $x_{i,j}$ and $y_{i,j}$ generate the group of modular symbols for 
                the Fermat group $\Phi(n)$. Furthermore, we know from Equation \ref{relation} that they 
                satisfy the relation
                 \begin{align} \label{x,y}
                         x_{i,j}+ y_{i,j}=x_{i+1,j} +y_{i+1,j-1}.
                 \end{align}
                Equation \ref{x,y} gives us a system of linear equations. Let $R$ be the $(n^2 \times 2n^2)$ 
                matrix obtained by 
                this system of equations with respect to the ordered basis  
                \[ 
                    \{x_{i,j},y_{i,j}: i,j=0,\ldots n-1 \}.
                \]
               Let $R^{i,j}$ denote the $i j$'th row of $R$, i.e., the row corresponding to the 
               equation: 
                         \[ 
                            x_{i,j}+ y_{i,j} - x_{i+1,j} - y_{i+1,j-1}=0.                       
                          \]
               Notice that for $i\leq n-2$ the first non-zero entry of each $R^{i,j}$ is 1 and it is the 
               coefficient of $x_{i,j}$. Assume $i=n-1$. For $j$ fixed,  we add $\sum_{i\neq (n-1)}{R^{i,j}}$ to 
               $R^{n-1,j}$ and denote the new row we obtain by $R_0^{n-1,j}$. Notice that the first $n^2$ 
               entry of $R_0^{n-1,j}$ is $0$ since
                       \[  
                         \sum_{i=0}^{n-1} {(x_{i,j}-x_{i+1,j})}=0.
                       \]
               The equation corresponding to $R_0^{n-1,j}$ gives us that 
                      \[ 
                          \sum_{i=0}^{n-1}{ (x_{i,j}+ y_{i,j} - x_{i+1,j} - y_{i+1,j-1})}=  \sum_{i=0}^{n-1} 
                          {\left(y_{i,j}-y_{i+1,j-1}\right)}=0
                      \]
               If $j \geq 1$, then the first non-zero term of $R_0^{n-1,j}$ is the coefficient of 
               $y_{0,j-1}$. Similarly if $j=0$, then it is the coefficient of $y_{0,0}$. Now let 
                   \[ 
                   R_1^{n-1,j}=\sum_{k\leq j}{R_0^{n-1,k}}
                   \]
               and replace $R_0^{n-1,j}$ by $R_1^{n-1,j}$. Assume $j\neq n-1$, then the first non-zero 
               entry of $R_1^{n-1,j}$ is $1$ and it is the 
               coefficient of $y_{0,j}$. If $j=n-1$, then $ R_1^{n-1,n-1}$ is the zero row since
                         \[ 
                         \sum_k\left(\sum_i{y_{i,k}-y_{i+1,k-1}}\right)=0.
                         \]
            To summarize, we showed that
                           \[ 
                  \{ y_{i,j} : \hspace{1mm} 1\leq i \leq (n-1), \hspace{1mm}  0\leq j \leq (n-1) \} 
                  \]
                  \[ 
                   \{x_{n-1,j} : \hspace{1mm} 0\leq j \leq (n-1)    \} 
                   \]
               and 
                 \[
                  [y_{0,n-1}] \]
              are the free variables and they generate $\M_2(\Phi(n))$.   
   \qed

            %The proof of the Theorem gives us that if $j\neq n-1$, then
                     %  \begin{equation}\label{identity} 
                          %   \sum_{k\leq j}\left(\sum_{i=0}^{n-1}  {\left(y_{i,k}-y_{i+1,k-1}\right)}\right)=
                            % \sum_{i=0}^{n-1}{(y_{i,n-1}-y_{i,j}})=0.  
                        % \end{equation}
        %We will need Equation \ref{identity} in Section \ref{monodromy}.   
        \subsection{Proof of Theorem~\ref{thm:modular symbols}}
    We know that we can identify the action of $A$ and $B$ on $X_{\Phi(n)}$ by the action of $\epsilon_0$ and $\epsilon_1^{-1}$ on the Fermat curve. We also notice that $\tau$ acts as  an order $3$ element on $X_{\Phi(n)}$ and that $\tau$ sends $0$ to $1$, $i\infty$ to $1$ and $i\infty$ to $0$. We compute that 
    \[ 
    \phi (\epsilon_0(x:y:z))=\epsilon_1(\phi(x:y:z)) \hspace{2mm} \text{ and} \hspace{2mm} \phi(\epsilon_1(x:y:z))=\epsilon_0^{-1}(\epsilon_1^{-1}(x:y:z)). 
    \]
  Comparing this to the relations $ \tau A=B^{-1}\tau$ and $\tau B^{-1}=A^{-1}B\tau$ in $\PSL_2(\Z)$, we find that the action of $\tau$ on $X_{\Phi(n)}$ corresponds to the action of $\phi$ on $F_n$. Hence the statement follows from Proposition~\ref{prop:modular symbols} and the relation given in \eqref{relation}.

%--------------------------------------------------------------------------------------------           
     \section{Homology of Fermat Curves} \label{sec:homofFermat}

%****************************************************************************
   In this section, we will describe the first integral homology group of Fermat curves in terms of 
   modular symbols. 
     
    \begin{prop}\label{prop:cuspidal}
           The homology group $H^1(X_{\Phi(n)},\Z)$ of $X_{\Phi(n)}$ is generated by
                   \[\gamma_{i,j}:= [A^i\tau]-[A^{i+1}B^{n-1}\tau]+ [A^{i+1}B^{j}\tau] -[A^iB^{j+1}\tau] \] 
           for $1\leq i \leq (n-2)$ and $0\leq j \leq (n-2)$. 
     \end{prop}
     
     \begin{proof}
             We first compute that  
                           \[
                            \tau \{0,i\infty\}=\{1,0\}. 
                            \] 
                            Then using the fact that $B$ stabilizes $0$, we  find that
                           \[
                           A^i  \{1,0\} -A^i B^{j} \{ 1,0 \}= \{ A^i.1, A^iB^{j}.1\} 
                           \] and 
                           \[
                           A^{i+1}B^{j-1} \{1,0\}-A^{i+1}B^{n-1} \{1,0\}= \{A^{i+1}B^{j-1}.1,A^{i+1}B^{n-1}.1\}.
                            \] 
             Using Lemma~\ref{lem:actionon1}, we see that 
                     \[
                                 \pi(A^{i+1}B^{n-1}.1)=\pi(A^i.1). 
                     \]
                  One can similarly show that   
                     \[ 
                     \pi(A^{i+1}B^{j-1}.1)= \pi(A^i B^j.1).
                     \]
                 Therefore we obtain that 
                 \[ 
                 \gamma_{i,j}= \{ A^i.1, A^iB^{j}.1\} +  \{A^{i+1}B^{j-1}.1,A^{i+1}B^{n-1}.1\} 
                 \] and that
                      \[ 
                      \pi(A^i.1)-\pi(A^i B^j.1) + \pi(A^{i+1}B^{j-1}.1) - \pi(A^{i+1}B^{n-1}.1)=0  
                      \] 
                  as a $0$-dimensional cycle on $X_{\Phi(n)}$.    
                  Hence by Theorem \ref{Manin2}, 
                           \[ 
                          \gamma_{i,j}= [A^i\tau]-[A^{i+1}B^{n-1}\tau]+  [A^{i+1}B^{j}\tau] -[A^iB^{j+1}\tau] 
                           \]  
                  is in $H^1(X_{\Phi(n)},\Z)$ for every $i,j$ such that $1\leq i\leq n-2$ and 
                  $0\leq j \leq n-2$. 
                  We will show that the set of $\gamma_{i,j}$ is $\Z$-linearly independent. 
                  
                  To ease the 
                  notation, we will use $y_{i,j}$ for $[A^iB^j\tau]$ as in the proof of Proposition~\ref{prop:modular symbols}.
                  Hence $\gamma_{i,j}$ denotes the cycle 
                  \[ 
                  y_{i,0}-y_{i+1,n-1}+y_{i+1,j}-y_{i,j+1}
                  \]
                   and suppose that 
                         \begin{equation} \label{linearity} 
                                  \sum_{i,j}{n_{i,j}\gamma_{i,j}}=0. 
                          \end{equation}
                  Here in this sum $i$ runs from $1$ to $n-2$ and $j$ runs from $0$ to $n-2$.
                  
                  We will prove by induction that $n_{i,j}=0$ for all $i,j$.
                  We know by Proposition~\ref{prop:modular symbols} that the set of $y_{i,j}$ (for $i=1,\ldots,n-1$ and $j=0,\ldots, n-1$) is $\Z$-linearly 
                  independent. If $j$ is fixed, then the coefficient of $y_{1,j+1}$ in the sum given in 
                  \eqref{linearity} is 
                  $-n_{1,j}$. Hence 
                    \[ 
                      n_{1,j}=0 \hspace{2mm} \text{for all} \hspace{2mm}  0\leq j \leq n-3.
                      \]
                  Similarly the coefficient of $y_{1,0}$ 
                  is $\sum_{j}{n_{1,j}}$ and thus 
                     \[ 
                     n_{1,n-2}=0.
                     \]
                  This proves the base step of the induction. Assume $n_{i-1,j}=0$ for all $j=0,\ldots,n-2$ and $i \geq 2$. If 
                  $j \geq 1$, then the coefficient of $y_{i,j}$ is 
                  $-n_{i,j-1}+n_{i-1,j}$ and using the induction hypothesis, we obtain
                    \[ 
                    n_{i,j-1}=0  \hspace{2mm} \text{for all} \hspace{2mm}  j=1,\ldots,n-2.
                    \]
                   For $j=0$, we look at the coefficient of $y_{i,0}$ which equals to 
                  $ n_{i-1,0} + \sum_{j}n_{i,j}$. Since we showed $n_{i,j}=0$ for every $j \leq n-3$ we see 
                  that $n_{i,n-2}$ must also equal to $0$. Hence the set
                  \[ 
                     \{ \gamma_{i,j}: \hspace{1mm} 1\leq i \leq n-2, \hspace{1mm} 0\leq j \leq n-2 \}
                  \] 
                  is linearly independent and it is a basis for the homology group since the genus of 
                 the Fermat curve is $(n-1)(n-2)/2$ which implies that the rank of the first homology group 
                 is $(n-1)(n-2)$.
          
       \end{proof}    
       
       \subsection{Proof of Theorem~\ref{thm:cuspidal}}
      
       Notice that the set $\{ \gamma_{i,j+1}-\gamma_{i,j} \mid 1\leq i\leq n-2, \hspace{2mm} 0\leq j \leq n-2\}$ also forms a basis for $H_1(F_n,\Z)$. Moreover using \eqref{relation} we find that 
       \[ \gamma_{i,j+1}-\gamma_{i,j}=[A^{i+1}B^{j+2}]-[A^iB^{j+2}]+[A^iB^{j+1}]-[A^{i+1}B^{j+1}].
       \]
(Notice that $\gamma_{i,n-1}=0$ for every $1\leq i \leq n-2$.) 

We translate this basis into $\{ \epsilon_0^i\epsilon_1^j (1-\epsilon_0)(1-\epsilon_1)\beta \}$ using the fact that the action of $A$ (resp. $B^{-1}$) on $F_n$ is given by $\epsilon_0$ (resp. $\epsilon_1$).

 %--------------------------------------------------------------------------------------------       
              
\section{Fermat Surfaces} \label{Fermat surfaces}
%**********************************************************************************
    \subsection{Monodromy} 
  %----------------------------------------------------------------------------------------  
    
    Let $X$ and $Y$ be algebraic varieties over $\mathbb{C}$. Suppose that $f: X\rightarrow Y$ is a 
    fibration. Let $y_0$ be in $Y$ and let $X_0$ be the fiber of $f$ at $y_0$. For $\gamma$ in 
    $\pi_1(Y,y_0)$, define $H :X_0\times [0,1]\rightarrow Y $ as 
        \[   
            H(x,t):=\gamma(t)
        \]
    for all $x\in X_0$. Then by the homotopy lifting property of $f$,  there exists a lift 
        \[ 
         \tilde{H}: X_0 \times [0,1] \rightarrow X
        \]
          as given in the following diagram (See \cite{Davis-Kirk} for more 
    details):
         
          \begin{equation*}
             \begin{tikzcd}
                X_0 \times \{0 \} \arrow{r}  \arrow{d}             &  X  \arrow{d}{f} \\
                X_0 \times [0,1] \arrow{ur}{\tilde{H}}  \arrow{r} {H} &Y \\
             \end{tikzcd}.
          \end{equation*}
    The commutativity of the above diagram induces a morphism  $X_0 \rightarrow X_0$ (at $t=1$) 
    which induces a homomorphism on the homology groups. Hence we obtain an action of the 
    fundamental group $\pi_1(Y,y_0)$ of $Y$ on $H_1(X_0,\Z)$:
                           \[ \rho :\pi_1(Y,y_0) \longrightarrow \text{Aut}(H_1(X_0,\Z))\]
    which we call the monodromy action. It is worth noting here that the induced map 
    $X_0\rightarrow X_0$ is a self homotopy equivalence and it is unique up to a homotopy.
   
 %-----------------------------------------------------------------------------------   
 \subsection{Description of the Fibration}
 %Consider the surface inside $\Pro^2 \times \mathbb{A}^1$ given by the equation
   %\[
  %Y: x^n+y^n=z^n(1-u^n).
   %\]
  %We can write a morphism $g: Y \to \mathbb{A}^1$ given by
     %     \[ 
        %      g: ((x:y:z),u) \mapsto u.  
           %   \]
  %The surface $Y$ fibers over $\A^1$ and the fibers are isomorphic to the Fermat curve of degree $n$ i.e., 
   %\[
      %         x^n+y^n=z^n. 
    %\]   
    
    Let $X$ denote the Fermat surface given by the equation 
              \[ 
              x^n+y^n+z^n=w^n .
              \]
     in $\Pro_{\mathbb{C}}^3$. Consider the rational map $f: X\rightarrow \Pro^1$ given by $
              f: (x:y:z:w) \mapsto (z:w).$
    We first blow up the variety $X$ successively at the points of the set 
            \[ 
                 B=\{(x:y:0:0) \in \Pro^3: \hspace{1mm} x^n+y^n=0 \}.
            \]
     We obtain a variety $\bar{X}$, a proper morphism  $\bar{f}: \bar{X} \rightarrow \Pro^1$ and a 
     birational map $\varphi: \bar{X} \rightarrow X$. A generic fiber of $f$ is given by the equation 
     $x^n+y^n=1$.    
     The fibers of $\bar{f}$ are projective curves since $\bar{f}$ is proper. Then a generic fiber of 
     $\bar{f}$ is isomorphic to the curve
              \[
               x^n+y^n=z^n. 
               \] 
    The singular fibers are given by $x^n+y^n=0$ when $z^n=w^n$. 
    The morphism $\bar{f}$ is flat since $k[x,y,z]/(x^n+y^n+z^n=1)$ is torsion free as a $k[z]$ module and $k[z]$ is a principal ideal domain. (The fact that $k[x,y,z]/(x^n+y^n+z^n-1)$ is torsion free follows since it is an integral domain.) 
   
    Let $U$ denote the variety 
    $\bar{f}^{-1}(\Pro^1-S)$ with 
              \[ 
              S=\{[\zeta_n^k:1] \in \Pro^1: k=0,\ldots,n-1\}.
              \]  
   Then we obtain a map $g: U \rightarrow V$ where $V=\Pro^1-S$. Hence we conclude that the map $g: U \to V$ induced by $\bar{f}$ is a surjective, proper, and flat morphism with smooth fibers which implies that $g$ is in fact smooth hence it is a submersion on the underlying manifolds of complex points.  By Ehresmann's theorem \cite{Ehresmann}, a proper submersion is a fibration (in fact a fiber bundle); hence $g:U \to V$ is a fibration. We will compute the monodromy of $g$.
     
     Let $y_0$ be the point $[0:1] \in \Pro^1$. To describe the monodromy action, we first formulate the generators of the fundamental group $\pi_1(V, y_0)$ in the following way: 
   
    We may take the leaves (around $n$th the roots of unity) of the polar rose for each $n$.
    When $n=3$, one can visualize these generators as in the figure below.
           \begin{figure}[ht]
              \centering
               \begin{tikzpicture}[scale=0.40] 
                 \begin{polaraxis}
                    \addplot[mark=none,domain=0:360,samples=300] {2*cos(3*x)};
                      \draw[black, thick]   (130,2) node {\LARGE $\sigma_1$};
                      \draw[black,thick]  (230,1.8) node{\LARGE $\sigma_2$};
                      \draw[black, thick] (20,1.8) node{\LARGE $\sigma_0$};
                 \end{polaraxis}
               \end{tikzpicture}
               \hspace{3mm}
               \caption{Generators of $\pi_1(V,y_0)$} \label{rose} in $w=1$ plane for $n=3$.
              \end{figure}
   %The diagram was drawn with the help of a answer on stackexchange.
   
    Let $\sigma_k: [0,1] \to \CC$ be the loop around the singular point $[\zeta_n^k:1]$ for $k=0,\ldots,n-1$ such that $\sigma_k(0)=\sigma_k(1)=0$. %Let $\rho$ denote the monodromy representation. 
    We would like to find a lift $\tilde{\sigma}_k$ for $\sigma_k$ such that $g(\tilde{\sigma}_k)=\sigma_k$. Since $\sigma_k(t)$ lies in the plane $w=1$, $\tilde{\sigma}_k(t)$ also should be in $w=1$ inside $U$. Hence we may assume that $w\neq 0$. Since $\bar{X}$ and $X$ are isomorphic on $w\neq 0$, we may assume that $U$ is given by the equation $x^n+y^n +z^n=1$ in the affine space $\A^3$. 
   Hence the fibration can be described by $U \to \A^1-\{{\zeta_n}^k \mid k=0,\ldots,n-1 \}$ where $(x,y,z) \mapsto z$ and the fiber of $z=0$ is given by the affine equation $x^n+y^n=1$. Remember that we pick the base point $y_0$ as $z=0$. 
      
  We notice first that if we rotate $\sigma_k $ by $\frac{2\pi k}{n}$ degree clockwise, we obtain $\sigma_0$, that is, $e^{\frac{-2\pi ki}{n}}(\sigma_k(t))=\sigma_0(t)$ for every $t \in [0,1]$ and $k=0,\ldots,n-1$. Hence $({\sigma_k}(t))^n=(\sigma_0(t))^n$ in $\CC$ for every $k=0,\ldots,n-1$. 
  
  Let $(x_0,y_0)$ be a point on the fiber $x^n+y^n=1$. Define
                \[ 
                 \tilde{\sigma}_k: [0,1] \rightarrow U 
                 \]
                \[ 
                \tilde{\sigma}_k(t)=(\lambda(t) x_0,\lambda(t) y_0, \sigma_k(t)) 
                \]
   where $\lambda(t)^n=1-\sigma_k(t)^n=1-\sigma_0(t)^n.$ In particular $\lambda$ does not depend on $k$. One may write $\lambda(t)^n$ in polar form as $r(t)e^{ig(t)}$. Then we define $\lambda(t)=\sqrt[n]{r(t)}e^{i\frac{g(t)}{n}}$ where $\sqrt[n]{-}$ denotes the real positive $n$th root and $\lambda(0)=1$. 
    
    The path $\tilde{\sigma}_k(t)$ takes values in $U$, that is, 
   \[ 
     (\lambda(t) x_0)^n+ (\lambda(t) y_0)^n + \sigma_k(t)^n=(1-\sigma_k(t)^n)(x_0^n+y_0^n)+\sigma_k(t)^n=1.
   \]
The path $\tilde{\sigma}_k$ begins at $\tilde{\sigma}_k(0)=(x_0,y_0,0)$ and ends at $\tilde{\sigma}_k(1)=(\zeta_n x_0,\zeta_n y_0,0) \in U(\CC)$. 
(This follows either using a direct computation of $\lambda$ or from Section~\ref{Automorphisms} viewing this action inside the automorphism group of the Fermat curve $x^n+y^n=z^n$. )
  
    Hence 
   $\tilde{\sigma}_k$ induces an action on the fiber $x^n+y^n=w^n$ as
                \[ 
                (x:y:w) \mapsto (\zeta_n x, \zeta_n y: w). 
                \] 
 In particular, this action is independent of $k$.
%------------------------------------------------------------------------------------- 
\section{Calculation of the Monodromy}\label{monodromy}
 %**************************************************************************
 %\iffalse   
   In Section~\ref{Fermat surfaces}, we showed that the action of each generator of the fundamental 
    group on a generic fiber is identical. %We also showed in section~\ref{Automorphisms} that the automorphism of $F_n$ we found above is given by the action of $\epsilon_0\epsilon_1$.
   We will compute the monodromy by describing the action 
    of $\epsilon_0\epsilon_1$ on the basis given in Theorem~\ref{thm:cuspidal}.  
    
    Let $s$ denote the element $(1-\epsilon_0)(1-\epsilon_1)\gamma$ and let $s_{i,j}$ denote the basis elements $\epsilon_0^i\epsilon_1^js$ for $i=1,\ldots,n-2$ and $j=0,\ldots,n-2$. 
    
   \begin{theorem}\label{thm:action}
     The action of $\epsilon_0$ (resp. $\epsilon_1$) on the basis $\{s_{i,j}\}$ of $H_1(F_n,\Z)$
      is given by \[  
      s_{i,j} \mapsto s_{i+1,j} \hspace{1mm} (\text{resp}. \hspace{1mm} s_{i,j} \mapsto s_{i,j+1} )\hspace{2mm}     \text{for} \hspace{2mm} 1\leq i\leq  n-2,\hspace{2mm}            0\leq j\leq n-2.  
      \] 
 where $\epsilon_0^{n-2}\epsilon_1^j$ and $\epsilon_0^i\epsilon_1^{n-1}$ for $0\leq i,j \leq n-2$ can be expressed in terms of the basis elements via the relations
 \[ 
 \epsilon_0^j(1+ \epsilon_1 +\ldots + \epsilon_1^{n-1})s=0, \hspace{2mm}  \epsilon_1^j(1+ \epsilon_0 +\ldots + \epsilon_0^{n-1})s=0, \hspace{2mm} 
 \text{and} \hspace{2mm}  
\epsilon_0^j (1+ \epsilon_0\epsilon_1 +\ldots + \epsilon_0^{n-1}\epsilon_1^{n-1})s=0.
 \]    
   \end{theorem} 
  \begin{proof}
The relation given in Theorem~\ref{thm:modular symbols} translates into $(1-\epsilon_0)\gamma=(\epsilon_0\epsilon_1-1)\bar{\gamma}.$
 We observe that 
    \begin{align} \label{rel:2} \sum_{i=0}^{n-1}\epsilon_0^is=0, \hspace{2mm} \sum_{i=0}^{n-1}\epsilon_1^is=0,\hspace{2mm} \text{ and}  \hspace{2mm} \sum_{i=0}^{n-1}(\epsilon_0\epsilon_1)^is=0. \end{align}
 In other words, $(1+\epsilon_k +\epsilon_k^2+\ldots+\epsilon_k^{n-1} )$ for $k=1,2$ and $(1+\epsilon_0\epsilon_1 +\ldots+\epsilon_0^{n-1}\epsilon_1^{n-1} )$ operate as zero on $H_1(F_n, \Z)$ since $s$ is the generator of $H_1(F_n,\Z)$ as a cyclic $\Z[\mu_n \times \mu_n]$ module. The first two relations in \eqref{rel:2} gives us that 
 $\epsilon_k^{n-1}=-\sum_{j=0}^{n-2}\epsilon_k^j$ for $k=0,1$. Moreover, using the third relation in \eqref{rel:2} and the action of $\epsilon_1$ on it, we find for $1 \leq j \leq n-2$ that
    \begin{align*} 
    \epsilon_1^js=-\sum_{i=1}^{n-1}\epsilon_0^i\epsilon_1^{i+j}s &=-\Large(\sum_{i=1}^{n-2}\epsilon_0^i\epsilon_1^{i+j}\Large)s - \epsilon_0^{n-1}\epsilon_1^{n-1+j} s \\
             &=  -\sum_{i=1}^{n-2}\epsilon_0^i\epsilon_1^{i+j}s + (\sum_{i=0}^{n-2}\epsilon_0^i)\epsilon_1^{j-1}s \\
             &= -\sum_{\substack{i=1\\ i+j\neq n-1}}^{n-2}\epsilon_0^i\epsilon_1^{i+j}s + ( \sum_{i=0}^{n-2}\epsilon_0^i)\epsilon_1^{j-1}s + (\sum_{k=0}^{n-2}\epsilon_1^k)\epsilon_0^{n-j-1}s \\
             &=-\sum_{\substack{i=1\\ i+j\neq n-1}}^{n-2}\epsilon_0^i\epsilon_1^{i+j}s + ( \sum_{i=1}^{n-2}\epsilon_0^i)\epsilon_1^{j-1}s +\epsilon_0^{n-j-1} \sum_{k=0}^{n-2}\epsilon_1^ks +\epsilon_1^{j-1}s
      \end{align*}
Hence we only need to express $s$ in terms of the basis elements. In addition, the system of equations
\[ \epsilon_0^j (1+ \epsilon_0\epsilon_1 +\ldots + \epsilon_0^{n-1}\epsilon_1^{n-1})s=0
\]
for $j=0,\ldots,n-2$ determines $s$ in terms of the basis elements.
  \end{proof}  

\subsection{Example}
 Let $n=3$. Then $\epsilon_0^2=-1-\epsilon_0$ and $\epsilon_1^2=-1-\epsilon_1$. We also compute as in the proof above that 
  \begin{align*} 
   \epsilon_1s &= \epsilon_0(1+\epsilon_1)s + \epsilon_0s+s=2\epsilon_0s+\epsilon_0\epsilon_1s +s. 
  \end{align*} 
Now we find $s$ from the following equations:
\begin{align*}
1+\epsilon_0\epsilon_1 + \epsilon_0^2\epsilon_1^2&=0 \\
\epsilon_0 + \epsilon_0^2\epsilon_1+\epsilon_1^2&=0
\end{align*}
which translate to \begin{align*}
\epsilon_0-\epsilon_0\epsilon_1 - 2\epsilon_1-1&=0 \\
\epsilon_0 + 2\epsilon_0\epsilon_1+\epsilon_1 +2&=0.
\end{align*}
We obtain that $s=-\epsilon_0s-\epsilon_0\epsilon_1s$ and $\epsilon_1s=\epsilon_0s$. Hence the action of $\epsilon_0\epsilon_1$ on $\{s_{1,0},s_{1,1}\}$ is given by:
\[ \epsilon_0\epsilon_1: s_{1,0}\mapsto -s_{1,0}-s_{1,1}, \hspace{2mm} \epsilon_0\epsilon_1:s_{1,1}\mapsto s_{1,0}. 
\]
Hence the monodromy of our family is given by  \[    \begin{bmatrix} 
           0 & 1\\ -1 & -1 
           \end{bmatrix}
         \] 
     which coincides with the local monodromy of an elliptic surface with a singularity of three lines intersecting at 
     a single point given in \cite{Kodaira1}, \cite{Kodaira2},\cite{Miranda}.

%\bibliography{mybib}{}
%\bibliographystyle{alpha}
\bibliographystyle{ams}

\end{document}